\documentclass{amsart}
\usepackage[OT4]{fontenc}
\usepackage{color}
\newtheorem{theorem}{Theorem}[section]
\newtheorem{lemma}[theorem]{Lemma}

\theoremstyle{definition}

\theoremstyle{remark}
\newtheorem{remark}[theorem]{Remark}

\numberwithin{equation}{section}

\def\Om{\Omega}

\begin{document}

\title{Rad\'{o}- type theorem for subharmonic and plurisubharmonic functions}

%    Information for first author
\author{S\l awomir Dinew}
%    Address of record for the research reported here
\address{Faculty of Mathematics and Computer Science,  Jagiellonian University 30-348 Krakow, {\L}ojasiewicza 6, Poland}

\email{slawomir.dinew@im.uj.edu.pl}
%    \thanks will become a 1st page footnote.
\thanks{The first Author was supported by the Polish National Science Centre grant 2017/26/E/ST1/00955.}

%    Information for second author
\author{\.Zywomir Dinew}
\address{Faculty of Mathematics and Computer Science,  Jagiellonian University 30-348 Krakow, {\L}ojasiewicza 6, Poland}
\email{zywomir.dinew@im.uj.edu.pl}

%    General info
\subjclass[2020]{Primary 32U30; Secondary 31B05, 32U40, 32D20, 32U15, 35D40, 28A05, 28A78, 26B05, 26A21}

\date{}

\keywords{subharmonic function, plurisubharmonic function, removable set, Hausdorff measure, viscosity theory, Borel set}

\begin{abstract} 
We  observe that a recent result by Gardiner and Sj\"odin, solving a problem of Kr\'{a}l on subharmonic functions, can be easily generalized to yield a somewhat stronger result. This can be combined with a viscosity technique of ours, which we slightly improve, to obtain Rad\'{o}- type theorems for  plurisubharmonic functions.   Finally, we study the Borel complexity of the critical set and the set where the gradient does not exist finitely of subharmonic functions and general real valued functions. 

\end{abstract}

\maketitle

\section{Introduction}

In 1924 T. Rad\'o proved the following remarkable theorem \cite{Ra}:
\begin{theorem}
Let $\Omega\subseteq \mathbb{C}$,  be an open set and $f$ be a continuous function in $\Omega$ which is holomorphic on $\Omega\setminus\lbrace f^{-1}(0)\rbrace$. Then $f$ is holomorphic in $\Omega$.
\end{theorem}

 A distinguishing feature of Rad\'{o}'s result  is that,  unlike other removable singularity results where the {\it singularity set} is required to be small in some sense, here no size bounds are a priori assumed. It is just the image of this set that has to be small.

This theorem has been generalized by many authors in various directions. In particular analogous result holds for multi-dimensional holomorphic functions. In a different direction Stout \cite{St} proved that if $f$ is continuous in $\Omega\subseteq \mathbb{C}$, holomorphic in $\Omega\setminus E$, for a relatively closed set $E$, not locally constant, and $f(E)$ is polar then $f$ extends holomorphically past $E$.

Analogous questions are meaningful not only for the class of holomorphic functions. The analogue of the Rad\'o theorem also holds for harmonic functions and even more generally for solutions to homogeneous uniformly elliptic second order PDEs, see \cite{Kr}. Even more generally Rad\'o theorem holds for certain classes of viscosity solutions to nonlinear elliptic PDEs, see \cite{T} and references therein.

In this note we focus our attention to the case of subharmonic and plurisubharmonic functions. As the examples
$$u(x_1,\cdots,x_n)=-|x_1|,\ v(z_1,\cdots,z_n)=-|Re(z_1)|$$
clearly show (this observation was made by Pokrovskiĭ in \cite{PP}) the direct analogue of Rad\'o's theorem fails for these classes- both examples are continuous, even Lipschitz, and subharmonic (resp. plurisubharmonic) except on a set that is sent by $u$ (resp. by $v$) to zero. It turns out that an additional smoothness  assumption is the key to obtain a positive result. Our main theorem, which mirrors the result in \cite{Kr} for harmonic functions, can be stated as follows:
\begin{theorem}\label{main}
Let $\Omega$ be open in $\mathbb R^{n}$ (respectively in $\mathbb C^{n}$) and $E\subseteq \Omega$ be a Borel set.	If $u\in C^{1,p}(\Omega),\, p\in(0,1]$ is  subharmonic (respectively plurisubharmonic) in some open neighborhood of $\Omega\setminus E$ and the Hausdorff measure $\mathcal H^{p}(u(E))=0$  then $u$ is actually subharmonic (respectively plurisubharmonic) in $\Omega$. If $u\in C^{1}(\Omega)$ then the same conclusion holds if $u(E)$ is at most countable. The results are optimal with respect to the size of the image of $E$.
\end{theorem}
  It is interesting to note that our theorem is {\it sharp} in both settings whereas in many other situations sharp results for subharmonic functions are in general not sharp for plurisubharmonic ones. A typical example is Lelong's theorem: (pluri)subharmonic functions which are locally bounded above extend through polar sets. This is optimal for subharmonic functions, whereas there are non-polar sets through which locally bounded above plurisubharmonic  functions do extend. In principle, this is because counterexamples tend to belong to the ''harmonic part'' of the theory of subharmonic functions. We wish also to point out that as a rule general results for removable sets of subharmonic functions are usually sharp. On the other hand there are quite a few removability results for plurisubharmonic functions for which optimality of assumptions remains open, see  \cite{HP}, \cite{Po}.

Note that Theorem \ref{main} settles the extension problem in the classes $C^{1,p}, 0\leq p\leq 1$. If the functions in question are $C^2$ or better the picture changes dramatically as our next result shows:
\begin{theorem}\label{c2}
Let $\Omega$ be open in $\mathbb R^{n}$ (respectively in $\mathbb C^{n}$) and $E\subseteq \Omega$ be a Borel set.	If $u\in C^{2}(\Omega)$ is  subharmonic (respectively plurisubharmonic) in some open neighborhood of $\Omega\setminus E$ then $u$ extends subharmonically (respectively plurisubharmonically) to the whole  $\Omega$ if and only if $u(E)$ has empty interior.
\end{theorem}
We wish to point out that $int u(E)=\emptyset$ is easily seen to be necessary as elementary examples show, but the other implication is somewhat tricky especially along the critical set $\lbrace z\in\Omega: \nabla u(z)=0\rbrace$. In fact we shall use a recent result of Gardiner and S\"odin \cite{GS} about extensions of subharmonic functions past critical sets. Another key tool in our approach comes from our previous work \cite{DD} where we proved that plurisubharmonic functions with subharmonic singularities across a Lebesgue measure zero set extend plurisubharmonically. Also throughout the paper we make use of the following simple but fundamental fact from calculus:

If a real valued function $u$ attains an extremal value at some point $x$ where the gradient of $u$ exists finitely (that is the partial derivatives of $u$ exist finitely) then $\nabla u(x)=0$, regardless of whether $u$ is differentiable at $x$ or not.

The note is organized as follows. In the next section we discuss various technical results of independent interest related to weaker notions of differentiability that we shall use later on. In Section 3 we recall and slightly generalize our findings from \cite{DD}. In Section 4 we thoroughly discuss the result of Gardiner and S\"odin from \cite{GS} providing yet another slight generalization. Section 5 is devoted to the proofs of Theorems \ref{main} and \ref{c2}.  Finally in Section 6 some applications are provided. 

\section{Critical sets and sets of finite existence of the gradient of discontinuous functions}\label{borelstuff}

In this section we study the Borel complexity of the critical set and the set of finite existence of the gradient  of discontinuous functions. Our main point of interest are (pluri)subharmonic functions, thus upper semicontinuous ones, yet we take the general perspective of real valued functions, as this seems to be a  neglected area of function theory, see \cite{D}.

Let $\Omega\subseteq \mathbb R^{n}$ be open and $u$ be a finite real valued  function on $\Omega$. We consider the set
$$E_{1}:=\{x\in\Omega: \nabla u(x) \text{ exists finitely and } \nabla u(x)=0\}.$$
This is the \textit{critical set} of $u$.

We consider also
$$E_{2}:=\{x\in\Omega: \nabla u(x) \text{ does not exists finitely}\},$$
that is the set where either some partial derivative of $u$ at $x$ equals $\pm\infty$ or the limit defining this partial does not exist.  

Subharmonic functions take the values in $[-\infty,\infty)$, so we have to take special care of the set $u^{-1}(-\infty)$. Observe that  $v:= e^{u}$ is a finite valued subharmonic  function, so the discussion below applies to it.  Let $E_1^{u}$, $E_2^{u}$ and $E_1^{v}$, $E_2^{v}$ be the corresponding sets for $u$ and $v$. It is immediate that
\begin{enumerate}
	\item $u^{-1}(-\infty)\subseteq E_2^{u}$
	\item ${\displaystyle u^{-1}(-\infty)=\bigcap_{j=1}^{\infty}\{x\in\Omega : u(x)<-j\}}$ is $G_{\delta}$
	\item $x\in E_{1}^{u}\implies x\in E_1^{v}$
	\item $x\in E_{1}^{v}$ and $u(x)>-\infty \implies x\in E_{1}^{u}$
	\item $x\in E_{2}^{u}$ and $u(x)>-\infty \implies x\in E_{2}^{v}$
	\item $x\in E_{2}^{v} \implies x\in E_{2}^{u}$
	\item it may happen that $x\in E_{2}^{u}$ and $x\not \in E_{2}^{v}$, take for example $u(z)=\log \Vert z\Vert^2$, but then $x\in E_{1}^{v}$. This is because the ratio $\frac{e^{u(x+t\vec{e}_{i})}}{t}$ takes both positive and negative values, unless  it vanishes identically. The latter is possible if $n\geq 3$ because then  the coordinate axes form a polar set. 
\end{enumerate}
Thus, $u^{-1}(-\infty)\subseteq E_{1}^{v}\cup E_{2}^{v}$ and hence $E_{1}^{u}= E_{1}^{v}\setminus  u^{-1}(-\infty)$ and $E_{2}^{u}= E_{2}^{v}\cup  u^{-1}(-\infty)$,
 so we may restrict ourselves to finite real valued functions. The Borel complexity does not change.

It is a classical result by Zahorski \cite{Za} and Brudno \cite{Bru} that for any finite real valued function of \textit{one variable} the set where the derivative is infinite is $F_{\sigma\delta}$ of measure zero, and where it does not exist either finitely or infinitely is $G_{\delta\sigma}$.  In particular, $E_2$ is Borel.

H\'ajek \cite{Ha} proved that for any  finite real valued function $u$ of \textit{one variable},  the function $$x\to\limsup_{t\to 0}\frac{u(x+t)-u(x)}{t},$$
called the \textit{upper derivative}, is a Baire- 2 function. 
Hence, the critical set $E_1$ of $u$, which is
$$\left\{x: \limsup_{t\to0} \frac{u(x+t)-u(x)}{t}\leq 0\right\}\bigcap\left\{x: \liminf_{t\to0} \frac{u(x+t)-u(x)}{t}\geq 0\right\},$$
is $F_{\sigma\delta}$, and hence Borel.

In higher dimensions these results do not generalize in the expected way. Actually, the best one can get for an arbitrary finite real valued function $u$ is that each one dimensional cross-section of  the sets where the $i$-th partial derivative vanishes and respectively does not exist finitely,  in direction parallel to the $i$-th coordinate axis is Borel. However,  the example of Serrin \cite{S} shows that both $E_1$ and $E_2$ may fail even to be Lebesgue measurable for a measurable function $u$. The example of Neubauer and Hahn \cite{N} shows that $E_1$ and $E_2$ may fail to be Borel for a function $u$ which is Baire- 3, yet $E_1$ and $E_2$ are Lebesgue measurable for any Baire function $u$. On the other hand, a recent result by Mykhaylyuk and Plichko \cite{MP} shows that for a continuous $u$ the set where the gradient exists finitely is $F_{\sigma\delta}$. In particular, $E_2$ is $G_{\delta\sigma}$. We prove the following result:

\begin{theorem}
	For any Baire- 1 (in particular semicontinuous) finite valued function $u:\Omega\to \mathbb R$ the sets $E_1$ and $E_2$, defined above, are Borel (both can be empty).
\end{theorem}

We begin with a lemma that, in a sense, saves the proof from the continuous case.
\begin{lemma}\label{fsigma}
	Let $I\subseteq \mathbb R$ be an interval, $B\subseteq \mathbb R^{n}$ be a $F_{\sigma}$ set. Then any set of the form $$\{(x,x_2,\ldots,x_n)\in\mathbb R\times \mathbb R^{n-1}:  x_1 - x\in I, (x_1,\ldots,x_n)\in B\}$$ is $F_{\sigma}$.  \end{lemma}

\begin{proof} As intervals are $F_{\sigma}$ sets, we observe that $B\times I\subseteq \mathbb R^{n+1}$ is $F_{\sigma}$. Now the mapping $f:\mathbb R^{n+1}\to\mathbb R^{n}$, where
	$$f(x_1,\ldots, x_{n+1})=(x_1-x_{n+1},x_2,\ldots,x_{n})$$
	is continuous, as it is a projection, and maps $B\times I$ to the set from the statement of the lemma. Continuous mappings  $\mathbb R^{n+1}\to \mathbb R^{n}$ preserve $F_\sigma$ sets. This can be seen as follows. Every $F_\sigma$ set is at most countable union of closed sets, and each closed set in $\mathbb R^{n+1}$ can be represented as the at most countable union of compact sets. As continuous mappings transform compact set to compact ones, we conclude by observing that $f(\bigcup_{s} A_s)=\bigcup_{s} f(A_s)$ for any family of sets $A_s$.
\end{proof}
\begin{remark}
	We remark that the lemma is specific for $\mathbb R^{n}$, as in general topological spaces it is not true that continuous mappings preserve $F_{\sigma}$ sets.
\end{remark}

For the part concerning $E_2$ we follow closely \cite{MP}. Let $A_{m,k}$ be the set of all $(x,y)\in \mathbb R\times\mathbb R^{n-1}$ such that for all $w,v\in \left(x-\frac{1}{k},x+\frac{1}{k}\right)$ one has 
$$|u(w,y)-u(v,y)|\leq \frac{1}{m}.$$
Let $B_{m,k}$ be the set of all $(x,y)\in \mathbb R\times\mathbb R^{n-1}$ such that for all $w,w'\in \left(x,x+\frac{1}{k}\right)$ and all $v,v'\in \left(x-\frac{1}{k},x\right)$  one has 
$$\left|\frac{u(w,y)-u(v,y)}{w-v}-\frac{u(w',y)-u(v',y)}{w'-v'}\right|\leq \frac{1}{m}.$$
Thus, $$A_{m,k}=\bigcap_{w,v \in \left(x-\frac{1}{k},x+\frac{1}{k}\right)}\left\{(x,y): |u(w,y)-u(v,y)|\leq \frac{1}{m}\right\},$$
$$B_{m,k}=$$
$$\bigcap_{v,v'\in \left(x-\frac{1}{k},x\right),\, w,w'\in \left(x,x+\frac{1}{k}\right)}\left\{(x,y): \left|\frac{u(w,y)-u(v,y)}{w-v}-\frac{u(w',y)-u(v',y)}{w'-v'}\right|\leq \frac{1}{m} \right\}.$$
As the restriction of $u$ to the last $(n-1)$ coordinates with the first coordinate fixed is a Baire- 1 function, so is the function $\mathbb R^{n-1}\ni y\to |u(w,y)-u(v,y)|\in \mathbb R$, hence the set $\{|u(w,y)-u(v,y)|\leq \frac{1}{m}\}$, for fixed $w$ and $v$, is $G_\delta$. The same holds with the set from the definition of $B_{m,k}$ with $v,v',w,w'$ fixed. The difference with the continuous case is that these sets need not be closed. Since in the definitions of $A_{m,k}$ and $B_{m,k}$ we take uncountable intersections it is not clear that these sets are Borel.

\begin{lemma} The sets $A_{m,k}$ and $B_{m,k}$ are $G_\delta$.	
\end{lemma}

\begin{proof} We note that the complement of $A_{m,k}$ satisfies
	$$A_{m,k}^{c}=\bigcup_{w,v \in \left(x-\frac{1}{k},x+\frac{1}{k}\right)}\left\{(x,y): u(w,y)-u(v,y)> \frac{1}{m}\right\}$$$$\bigcup \bigcup_{w,v \in \left(x-\frac{1}{k},x+\frac{1}{k}\right)}\left\{(x,y): u(w,y)-u(v,y)<- \frac{1}{m}\right\}. $$
	It is enough to  prove that $\bigcup_{w,v \in \left(x-\frac{1}{k},x+\frac{1}{k}\right)}\left\{(x,y): u(w,y)-u(v,y)> \frac{1}{m}\right\}$ is $F_\sigma$, as the proof for $\bigcup_{w,v \in \left(x-\frac{1}{k},x+\frac{1}{k}\right)}\left\{(x,y): u(w,y)-u(v,y)<- \frac{1}{m}\right\}$ is the same.
	Observe that
	 $$\bigcup_{w,v \in \left(x-\frac{1}{k},x+\frac{1}{k}\right)}\left\{(x,y): u(w,y)-u(v,y)> \frac{1}{m}\right\}$$
$$=\bigcup_{r\in \mathbb Q}$$
$$\left[\bigcup_{w \in \left(x-\frac{1}{k},x+\frac{1}{k}\right)}\left\{(x,y): u(w,y)> r\right\}\bigcap \bigcup_{v \in \left(x-\frac{1}{k},x+\frac{1}{k}\right)}\left\{(x,y): r-u(v,y)> \frac{1}{m}\right\}\right].$$
	 Now the problem is reduced to proving that $\bigcup_{w \in \left(x-\frac{1}{k},x+\frac{1}{k}\right)}\left\{(x,y): u(w,y)> r\right\}$ is $F_\sigma$. This follows from Lemma \ref{fsigma} with $B=\{(x,y): u(x,y)> r\}$ which is $F_{\sigma}$, as $u$ is Baire- 1, and $I=\left(-\frac{1}{k},\frac{1}{k}\right)$.
	 
	 The proof for $B_{m,k}$ is similar but longer. Again,
	 $$B_{m,k}^{c}=D\bigcup F,$$
where $D$ is equal to
$$\bigcup_{v,v'\in \left(x-\frac{1}{k},x\right),\, w,w'\in \left(x,x+\frac{1}{k}\right)}\left\{(x,y): \frac{u(w,y)-u(v,y)}{w-v}-\frac{u(w',y)-u(v',y)}{w'-v'}> \frac{1}{m} \right\}$$
and $F$ is equal to 
$$\bigcup_{v,v'\in \left(x-\frac{1}{k},x\right),\, w,w'\in \left(x,x+\frac{1}{k}\right)}\left\{(x,y): \frac{u(w,y)-u(v,y)}{w-v}-\frac{u(w',y)-u(v',y)}{w'-v'}<- \frac{1}{m} \right\}.$$
The set $D$ can be written as 
	 $$\bigcup_{v,v'\in \left(x-\frac{1}{k},x\right),\, w,w'\in \left(x,x+\frac{1}{k}\right)}\left\{(x,y): \frac{u(w,y)-u(v,y)}{w-v}-\frac{u(w',y)-u(v',y)}{w'-v'}> \frac{1}{m} \right\}$$$$=\bigcup_{r\in \mathbb Q} \left[ \bigcup_{v\in \left(x-\frac{1}{k},x\right),\, w\in \left(x,x+\frac{1}{k}\right)}\left\{(x,y): \frac{u(w,y)-u(v,y)}{w-v}> r \right\}\right.$$
	 $$\left.\bigcap \bigcup_{v'\in \left(x-\frac{1}{k},x\right),\, w'\in \left(x,x+\frac{1}{k}\right)}\left\{(x,y): r-\frac{u(w',y)-u(v',y)}{w'-v'}> \frac{1}{m} \right\}\right].$$
	 Further, 
	 $$\bigcup_{v\in \left(x-\frac{1}{k},x\right),\, w\in \left(x,x+\frac{1}{k}\right)}\left\{(x,y): \frac{u(w,y)-u(v,y)}{w-v}> r \right\}$$$$=\bigcup_{v\in \left(x-\frac{1}{k},x\right),\, w\in \left(x,x+\frac{1}{k}\right)}\left\{(x,y): u(w,y)-rw>u(v,y)-rv \right\}$$
	 $$=$$
$$\bigcup_{\rho\in\mathbb  Q}\left[ \bigcup_{ w\in \left(x,x+\frac{1}{k}\right)}\left\{(x,y): u(w,y)-rw>\rho \right\}\bigcap \bigcup_{v\in \left(x-\frac{1}{k},x\right)}\left\{(x,y): \rho>u(v,y)-rv \right\}\right].$$
	 Finally, $\bigcup_{ w\in \left(x,x+\frac{1}{k}\right)}\left\{(x,y): u(w,y)-rw>\rho \right\}$ is $F_{\sigma}$ by Lemma \ref{fsigma}, where $B= \left\{(x,y): u(x,y)-rx>\rho \right\}$ is $F_{\sigma}$ as $u(x,y)-rx$ is Baire- 1, and $I= \left(0,\frac{1}{k}\right)$.
	 
	\end{proof}  
\begin{remark}
	As above, this lemma is specific for $\mathbb R^{n}$.
\end{remark}
 
Now $A_{m,k}$ is $G_\delta$, hence $A=\bigcap_{m=1}^{\infty}\bigcup_{k=1}^{\infty} A_{m,k}$ is $G_{\delta\sigma\delta}$. Likewise $B_{m,k}$ is $G_\delta$ hence $B=\bigcap_{m=1}^{\infty}\bigcup_{k=1}^{\infty} B_{m,k}$ is $G_{\delta\sigma\delta}$. The set of points where $\frac{\partial u}{\partial x_1}$ exists and is finite is exactly $A\cap B$. The same is true for all the other partial derivatives, hence the set, where $u$ allows a gradient is $G_{\delta\sigma\delta}$. Hence, $E_2$ is $F_{\sigma\delta\sigma}$.

 For the part concerning $E_1$ we follow closely \cite{N}.
 We denote by
 $$\underline{\frac{\partial u}{\partial x_1}}(x,y):=\liminf_{t\to 0}\frac{u(x+t,y)-u(x,y)}{t},\quad \overline{\frac{\partial u}{\partial x_1}}(x,y):=\limsup_{t\to 0}\frac{u(x+t,y)-u(x,y)}{t} $$
 the \textit{lower} and respectively \textit{upper partial derivatives} of $u$ with respect to $x_{1}$ at $(x,y)\in \mathbb R\times\mathbb R^{n-1}$. These functions may take infinite values. What follows is in principle the same as in \cite{N}, except for the notation, the fact that in \cite{N} one-sided lower partial derivatives are considered and the lemma below.
 
 We define 
 $$A_p:=\left\{(x,y)\in \Omega: \underline{\frac{\partial u}{\partial x_1}}(x,y)<p\right\}.$$
 Without loss of generality we may assume that $p>1$, as $A_p$ for the function $u$ is the same as $A_{q+p}$ for the function $u(x,y)+qx$ which is of the same class as $u$.
 Further, we define
 $$U^{+}_{r,k}:=\left\{(x,y): u(x+h,y)<r \text{ for some } h\in \Big[ \frac{1}{k+1}, \frac{1}{k}\Big)\right\}$$$$=\bigcup_{h\in \big[ \frac{1}{k+1}, \frac{1}{k}\big) }\{(x,y): u(x+h,y)<r\}$$
 $$=\left\{(z,y)\in\mathbb R\times \mathbb R^{n-1}:  z - x\in  \Big[ \frac{1}{k+1}, \frac{1}{k}\Big), (x,y)\in \{u(x,y)<r\}\right\},$$
 $$U^{-}_{r,k}:=\left\{(x,y): u(x+h,y)>r \text{ for some } h\in \Big(-\frac{1}{k},-\frac{1}{k+1}\Big] \right\}$$$$=\bigcup_{h\in \big(-\frac{1}{k},-\frac{1}{k+1}\big] }\{(x,y): u(x+h,y)>r\}$$ 
 $$=\left\{(z,y)\in\mathbb R\times \mathbb R^{n-1}:  z - x\in \Big(-\frac{1}{k},-\frac{1}{k+1}\Big], (x,y)\in \{u(x,y)>r\}\right\},$$
 $$V^{+}_{r,k,p}:=\left\{(x,y): u(x,y)>r-\frac{p}{k}\right\},$$
 $$V^{-}_{r,k,p}:=\left\{(x,y): u(x,y)<r+\frac{p}{k}\right\}.$$
 As $u$ is Baire- 1, $V^{+}_{r,k,p}$, as well as $V^{-}_{r,k,p}$,  are $F_{\sigma}$.
 We further define
 $$W_{k,p}:=\left(\bigcup_{r\in \mathbb Q} U^{+}_{r,k}\cap V^{+}_{r,k,p}\right)\bigcup \left(\bigcup_{r\in \mathbb Q} U^{-}_{r,k}\cap V^{-}_{r,k,p}\right)$$
 $$=\left\{(x,y): u(x+h,y)<u(x,y)+\frac{p}{k} \text{ for some } h\in  \Big[ \frac{1}{k+1}, \frac{1}{k}\Big)\right\}$$
 $$\bigcup\left\{(x,y): u(x+h,y)>u(x,y)-\frac{p}{k} \text{ for some } h\in \Big(-\frac{1}{k},-\frac{1}{k+1}\Big] \right\}.$$
 Let $$\varphi(h):=\begin{cases}h\left(\left \lceil\frac{1}{h}\right \rceil-1\right), & \text { if } h>0\\ h\left(\left \lfloor\frac{1}{h}\right \rfloor+1\right), & \text { if } h<0\end{cases}.$$ Clearly, $\lim_{h\to 0}\varphi(h)=1$.
 
  Further, we set
 $$Z_{k,p}:= \left\{(x,y): \frac{u(x+h,y)-u(x,y)}{h}<\frac{p}{\varphi(h)} \text{ for some } -\frac{1}{k}<h<\frac{1}{k}, h\neq 0\right\}$$ 
 $$=\bigcup_{h\in\left(-\frac{1}{k},\frac{1}{k}\right)\setminus\{0\}}\left\{(x,y): \frac{u(x+h,y)-u(x,y)}{h}<\frac{p}{\varphi(h)} \right\}.$$
 Now $Z_{k,p}=\bigcup_{j=k}^{\infty} W_{j,p}$. Let $B_p:=\bigcap_{k=1}^{\infty} Z_{k,p}$. Finally,
 $$A_p$$
is equal to
$$\bigcup_{m=1}^{\infty}B_{p-\frac{1}{m}}=\bigcup_{m=1}^{\infty}\bigcap_{k=1}^{\infty}\bigcup_{j=k}^{\infty}\left[\left(\bigcup_{r\in \mathbb Q} U^{+}_{r,j}\cap V^{+}_{r,j,p-\frac{1}{m}}\right)\bigcup \left(\bigcup_{r\in \mathbb Q} U^{-}_{r,j}\cap V^{-}_{r,j,p-\frac{1}{m}}\right)\right].$$
 
Now $U^{+}_{r,j}$ is exactly the type of set as in Lemma \ref{fsigma}, as $B=\{(x,y): u(x,y)<r\}$ is $F_\sigma$, and we can take $I=\Big[\frac{1}{k+1},\frac{1}{k}\Big)$. Hence, $U^{+}_{r,j}$ is $F_\sigma$. So are the sets $U^{-}_{r,j}$, $V^{+}_{r,j, p-\frac{1}{m}}$, $V^{-}_{r,j, p-\frac{1}{m}}$, and $\bigcup_{j=k}^{\infty}\left[\left(\bigcup_{r\in \mathbb Q} U^{+}_{r,j}\cap V^{+}_{r,j,p-\frac{1}{m}}\right)\bigcup \left(\bigcup_{r\in \mathbb Q} U^{-}_{r,j}\cap V^{-}_{r,j,p-\frac{1}{m}}\right)\right]$. Thus, $A_p$ is $F_{\sigma\delta\sigma}$, hence its complement $\left\{(x,y): \underline{\frac{\partial u}{\partial x_1}}(x,y)\geq p\right\}$ is $G_{\delta\sigma\delta}$. Similar reasoning applies to the upper partial derivative with respect to $x_1$.

Now the set where $\frac{\partial u}{\partial x_1}(x,y)= 0$ is $$\left\{(x,y): \underline{\frac{\partial u}{\partial x_1}}(x,y)\geq 0\right\}\cap \left\{(x,y): \overline{\frac{\partial u}{\partial x_1}}(x,y)\leq 0\right\},$$
which is $G_{\delta\sigma\delta}$. The same is true for all the other partial derivatives, hence the set $E_1$ is $G_{\delta\sigma\delta}$. Meanwhile, we also obtained that $\overline{\frac{\partial u}{\partial x_1}}$
 and $\underline{\frac{\partial u}{\partial x_1}}$ are Baire- 3 functions.

\begin{remark} We suspect that in fact $E_1$ is   $F_{\sigma\delta}$ and $E_2$ is $G_{\delta\sigma}$ at least for subharmonic functions $u$, as suggested by the special cases of continuous  \cite{MP}, and one variable functions  \cite{Ha}. If true, this would be optimal.\end{remark}

\section{Plurisubharmonic extension of subharmonic functions}

In this section we slightly generalize the main result of \cite{DD} that subharmonic functions which are plurisubharmonic outside a small set are actually plurisubharmonic.

Let $\Omega\subseteq \mathbb C^{n}$ be open and $u$ be a subharmonic function on $\Omega$. In addition to the sets $E_1$ and $E_2$  from Section \ref{borelstuff}, we consider also $E_{3}$ -  an arbitrary subset of $\Omega$ of zero Lebesgue measure which is disjoint  with $E_{1}\cup E_{2}$. Note that $E_3$ need not be Borel.

It is well known that any subharmonic function has a finite gradient almost everywhere, see e.g. \cite{Kr}.  This is somehow counter intuitive, as a subharmonic function  is merely upper semicontinuous, and the latter functions are  a priori  only continuous on a residual set which may happen to be of Lebesgue measure zero. More accurately, it is $e^{u}$ that is continuous on the residual set, as a subharmonic function can have a dense set of poles, and hence be nowhere continuous in the classical sense.   We note that, despite allowing a finite gradient on a big set, a subharmonic function, even a finite valued one, may be nowhere differentiable. All this allows us to assume that $E_2$ can be considered jointly with $E_3$ as \newline $\lambda^{2n}(E_2\cup E_3)=0$.

Let $E:=E_1\cup E_2\cup E_3$.
\begin{theorem}\label{DDD} Let $u$ be a subharmonic function in $\Omega$.
 Then if  $u$ is plurisubharmonic in some open neighborhood of $\Om\setminus E $ it is actually plurisubharmonic in the whole $\Om$. 
\end{theorem}

\begin{remark}
	This theorem is only interesting for functions $u$ with a big critical set. If $E_1$ and hence $E $ is of Lebesgue measure zero, then Theorem \ref{DDD} is a corollary of the main result in \cite{DD}. However, even for smooth  $u$ the critical set of $u$ may be quite big, and the Morse-Sard theorem specifies the minimal regularity of $u$ that guarantees that at least the image $u(E_1)$ is small. Typically, we will think of $E_{1}$ as some nowhere dense Cantor- type set of positive Lebesgue measure.  
\end{remark}

As in \cite{DD}, we recall  (see \cite{H}, Proposition 3.2.10') that subharmonicity and plurisubharmonicity is equivalent to subharmonicity and respectively plurisubharmonicity in the viscosity sense. 
\begin{lemma}\label{visc}
	Let $\Om\subseteq \mathbb C^{n}$ be open. An upper semicontinuous function $u$ on $\Om$ is subharmonic (respectively plurisubharmonic) if for every $z_0\in\Om$ and every local  function $\varphi \in C^2$ defined near $z_0$
	and satisfying $\varphi(z)\geq u(z)$ with equality at $z_0$ we have $\Delta\varphi(z_0)\geq 0$ (respectively we have $\frac{\partial ^2 \varphi}{\partial z_j\partial\bar{z}_k} (z_0)\geq 0$, that is the complex Hessian is non negative definite).  
\end{lemma}

For some points $z_0\in\Om$ such a function $\varphi$ may not exist. It is known the set of points allowing the local $C^2$ majorant is dense
 for general upper semicontinuous $u$, but may well be countable. This is, however, enough to recover (pluri)subharmonicity. On the other hand for (pluri)subharmonic functions  the set of points allowing $\varphi$ is of full Lebesgue measure. 

Also note that in the language of viscosity theory the above lemma says that $u$ is subharmonic iff it is a viscosity subsolution to the Laplace equation $\Delta v=0$ (that is $\Delta u\geq 0$ in the viscosity sense) and $u$ is plurisubharmonic iff it is a viscosity subsolution to the {\it constrained complex Hessian} equation $\det^{+}\left(\frac{\partial ^2 v}{\partial z_j\partial\bar{z}_k}\right)= 0$ (that is 
$\det^{+}\left(\frac{\partial ^2 u}{\partial z_j\partial\bar{z}_k}\right)\geq 0$ in the viscosity sense), where 
$${\det}^{+}(A)=\begin{cases}
\det(A)\ & \text{ if }\ A\geq 0;\\
-\infty\ & \text{ otherwise }.
\end{cases}$$
For more details regarding the viscosity theory of such constrained complex Hessian equations we refer to \cite{EGZ} and \cite{Ze}.

As a direct corollary of Lemma \ref{visc} plurisubharmonicity of $u$ would follow if one can show that for any $z_0\in E $ and any local $C^2$ majorant $\varphi\geq u$, $\varphi(z_0)=u(z_0)$ one has 
$$\frac{\partial ^2 \varphi}{\partial z_j\partial\bar{z}_k} (z_0)\geq 0$$
as matrices.

To this end we closely follow the argument from \cite{DD}.

Suppose $u$ allows a local $C^2$ majorant at $z_0\in E $. Translating if necessary one may assume that $z_0$ is the origin, that $\Om$ contains a ball $B_{\delta_0}$ centered at the origin and that $\varphi$ is defined on $B_{\delta_0}$. For a fixed $0<\delta<\frac{\delta_0}2$ we consider the function

$$v_{\delta}(z):=\varphi(z)+\delta\Vert z\Vert ^2-\delta^3-u(z).$$
By the very definition of $\varphi$ we have $v_{\delta}(z)\geq 0$ on the collar $B_{2\delta}\setminus B_{\delta}$. Recall (see \cite{DD}), that $v_{\delta}$ is bounded below on $B_{\delta}$ and $v_{\delta}$ is a lower semicontinuous viscosity supersolution to the Poisson equation
$\Delta v= \Delta\varphi+4n\delta=:f$, that is
$$\Delta v_{\delta}\leq f.$$

Consider the following convex envelope in $B_{2\delta}$
$$\Gamma_{v_\delta}(z):=\sup\lbrace l(z)|\ l-\text{affine},\, l\leq v_{\delta}\, \text{ on }\  B_{\delta},\, l\leq 0\ \text{ on }\ B_{2\delta}\setminus B_{\delta} \rbrace.$$
As $v_\delta$ is bounded below, the family of functions $l$ above is non void.

Just as in \cite{DD}, we exploit the Alexandrov-Bakelman-Pucci estimate  for viscosity supersolutions (see Theorem 3.2 in \cite{CC} for continuous supersolutions and \cite{I} for merely lower semicontinuous supersolutions):

\begin{lemma}\label{ABP}
	Let $v_\delta$ and $\Gamma_{v_{\delta}}$ be as above. Then there is a universal constant $C$, which  depends only on $n$ such that
	$$\delta^3\leq \sup_{B_{\delta}}v_{\delta}^{-}\leq C\delta \left(\int_{\lbrace B_{\delta}\cap \lbrace v_{\delta}=\Gamma_{v_{\delta}}\rbrace\rbrace}\max\{f,0\}^{2n}d\lambda^{2n}\right)^{\frac1{2n}}.$$
\end{lemma}

The upshot is that for every $\delta\in \left(0,\frac{\delta_0}2\right)$the function $v_\delta$ matches its convex envelope $\Gamma_{v_{\delta}}$ on a set of positive measure within $B_{\delta}$.

There are two possibilities. In the first case there exists a sequence $\delta_{j}\searrow 0$ such that  that there is a point  $z_{\delta_{j}}\in B_{\delta_{j}}\cap \lbrace v_{\delta_{j}}=\Gamma_{v_{\delta_{j}}}\rbrace\setminus E $. Then, arguing as in \cite{DD}, we find that the smallest eigenvalue of the complex Hessian of $\varphi$ at $z_{\delta_{j} }$ is at least $-\delta_{j}$, meaning that the complex Hessian of $\varphi$ at the origin is non negative definite.

In the second case there is a fixed $\delta_1<\frac{\delta_0}2$  such that for all $0<\delta<\delta_1$ one has $B_{\delta}\cap \lbrace v_{\delta}=\Gamma_{v_{\delta}}\rbrace\subseteq E $.
As $B_{\delta}\cap\lbrace v_{\delta}=\Gamma_{v_{\delta}}\rbrace$ has positive Lebesgue measure, and $\lambda^{2n}(E_2\cup E_3)=0$, it follows that $B_{\delta}\cap\lbrace v_{\delta}=\Gamma_{v_{\delta}}\rbrace\cap E_1$ has positive measure.

We observe that on $E_1\cap \lbrace v_{\delta}=\Gamma_{v_{\delta}}\rbrace$ we have $\nabla \Gamma_{v_{\delta}}=\nabla v_{\delta}$, because $v_{\delta}-\Gamma_{v_{\delta}}$ attains a minimum, and $\nabla v_{\delta}=\nabla \psi_{\delta}$, where $\psi_{\delta}(z):= \varphi(z) +\delta\Vert z\Vert ^2$, because $\nabla u$ vanishes.

The following condition (''monotonicity of the gradient'') is equivalent to convexity for $C^1$ functions: For any $z,w\in B_{\delta}$ (treated as real points in $\mathbb R^{2n}$) one has
$$\langle \nabla \Gamma_{v_{\delta}}(z)-\nabla\Gamma_{v_{\delta}}(w),z-w\rangle\geq 0.$$
Thus, on $B_{\delta}\cap\lbrace v_{\delta}=\Gamma_{v_{\delta}}\rbrace\cap E_1$ we have $\langle \nabla \psi_{\delta}(z)-\nabla \psi_{\delta}(w),z-w\rangle\geq 0.$

As $B_{\delta}\cap\lbrace v_{\delta}=\Gamma_{v_{\delta}}\rbrace\cap E_1$ has positive Lebesgue measure, we pick a point of density $z_\delta$ in it, meaning that $$\lim_{\varepsilon\to 0^{+}}\frac{\lambda^{2n}( B_{\varepsilon}(z_{\delta})\cap B_{\delta}\cap \lbrace v_{\delta}=\Gamma_{v_{\delta}}\rbrace\cap E_1 )}{\lambda^{2n}(B_{\varepsilon})}=1.$$ We argue that the real Hessian of $\varphi +\delta\Vert z\Vert ^2$ is non negative definite at $z_\delta$. If this is not the case then there is a set $H$ of directions $\mathbb R^{2n}\ni v=(v_1,\ldots,v_{2n})\neq 0$ such that $$\sum_{j,k=1}^{2n}\frac{\partial ^2\psi_{\delta}}{\partial x_j\partial x_k} (z_{\delta})v_jv_k< c\Vert v\Vert^2$$ for a fixed negative $c$ strictly between the smallest eigenvalue of the Hessian and $0$.  Moreover, the spherical projection of this set, which we call $A$, on $S^{2n-1}$  is  of positive $2n-1$- dimensional spherical measure $\sigma^{2n-1} (A)>0$.  There are directions in  $H\cap B_{\delta}\cap\lbrace v_{\delta}=\Gamma_{v_{\delta}}\rbrace\cap E_1$ arbitrarily close to $z_\delta$, as otherwise the density at $z_{\delta}$ would be no greater than $\frac{\sigma^{2n-1}(S^{2n-1})-\sigma^{2n-1}(A)}{\sigma^{2n-1}(S^{2n-1})}<1$. We pick \newline $z_{m}\in H\cap B_{\delta}\cap\lbrace v_{\delta}=\Gamma_{v_{\delta}}\rbrace\cap E_1,\, m=1,2,\ldots$ such that $z_{m}\to z_\delta$ and put ${v_{m}}=z_{m}-z_{\delta}$.  Using that for $C^2$ functions the Hessian is the Jacobian of the gradient, we get 
%$$\nabla \psi_{\delta}(z_j)= \nabla \psi_{\delta}(z_\delta)+ \sum_{j,k=1}^{2n}\frac{\partial ^2\psi_{\delta}}{\partial x_j\partial x_k} (z_{\delta})(z_j-z_\delta)+o(\Vert z_j-z_\delta\Vert),$$
$$\nabla \psi_{\delta}(z_m)= \nabla \psi_{\delta}(z_\delta)+ Jac(\nabla \psi_{\delta})(z_m-z_\delta)+o(\Vert z_m-z_\delta\Vert),$$ so we have
$$0>\frac{c}{2}>\frac{\sum_{j,k=1}^{2n}\frac{\partial ^2\psi_{\delta}}{\partial x_j\partial x_k} (z_{\delta})(v_m)_j(v_m)_k+o(\Vert {v_{m}}\Vert^2)}{\Vert{v_{m}}\Vert^2}$$
$$=\frac{\langle \nabla \psi_{\delta}(z_m)-\nabla \psi_{\delta}(z_{\delta}),z_{m}-z_{\delta}\rangle}{\Vert z_{m}-z_{\delta}\Vert^2}\geq 0.$$
  
When the real Hessian of a $C^2$ function is non negative definite then also the complex Hessian is non negative definite. Thus, the smallest eigenvalue of the complex Hessian of $\varphi$ at $z_\delta$ is at least $-\delta$. So again when $\delta\searrow 0$ we have $z_\delta\to 0=z_0$, and  again the complex Hessian of $\varphi$ at $z_0$ is non negative definite.

%The upshot is that for every $\delta\in \left(0,\frac{\delta_0}2\right)$ the gradient image of the set $\lbrace B_{\delta}\cap \lbrace v_{\delta}=\Gamma_{v_{\delta}}\rbrace$ through the function $\Gamma_{v_{\delta}}$ has positive measure. As $G$ is of measure zero, there is a point 
%$z_\delta\in\lbrace B_{\delta}\cap \lbrace v_{\delta}=\Gamma_{v_{\delta}}\rbrace\rbrace\setminus (\nabla \Gamma_{v_{\delta}})^{-1}(G) =\lbrace B_{\delta}\cap \lbrace v_{\delta}=\Gamma_{v_{\delta}}\rbrace\rbrace\setminus E_1$.

\section{Subharmonic and plurisubharmonic extension through the critical set}

   The key observation for this section is by Gardiner and S\"odin \cite{GS} (for harmonic functions this is due to Kr\'{a}l \cite{Kr}). It says that if $u$: $\Omega\rightarrow \mathbb{R}$ is $C^{1}$ on an open set $\Omega\subseteq \mathbb{R}^{n}, n\geq 2$ and subharmonic on $\{x\in\Omega:\nabla u(x)\neq 0\}$, then $u$ is subharmonic on all of $\Omega$. 
   
   We observe that the condition can be relaxed as follows.
   \begin{theorem}\label{GardinerSodin} Let $\Omega\subseteq \mathbb R^{n}$ be open, and let $u$ be a upper semicontinuous function on $\Omega$. Let $E_{1}\subseteq \Omega$ be the set where  $\nabla u$ exists and $\nabla u=0$. If $u$ is subharmonic on some open neighborhood of $\Omega\setminus E_1$ then $u$ is subharmonic on $\Omega$. 
   \end{theorem}

\begin{proof} We follow closely \cite{GS} in what follows and modify the argument only slightly to cover the more general setting. Since $u$ allows a gradient on $E_{1}$ it is finite valued there. As $u$ is subharmonic outside $E_{1}$, we have $u(x)<\infty$ for each $x \in \Omega$.
	
	As $u$ is upper semicontinuous, it is a pointwise limit of a decreasing sequence of continuous functions $u_{j}$.  Let $\varepsilon>0$ and $B$ be the open ball $\{x:\Vert x-x_{1}\Vert<r\}$ such that $\overline{B}\subseteq\Omega$. Let $v_{j}$ be the Poisson integral of $u_{j}$ in $B$	
	$$v_{j}(y)=\frac{1}{\sigma^{n-1}(S^{n-1})}\int_{\partial B}\frac{r^2-\Vert y-x_1\Vert^2}{r\Vert x-y\Vert^{n}}u_{j}(x)\, d\sigma^{n-1}(x).$$ We define 
	$$h_{\varepsilon, j}(x)=v_{j}(x)+\varepsilon\left(1+\frac{r^{2}-\Vert x-x_{1}\Vert^{2}}{2n}\right).
	$$
	The function $h_{\varepsilon, j}$ satisfies $h_{\varepsilon, j}\in C(\overline{B})\cap C^{\infty}(B)$, and solves the Dirichlet problem
	$$\left\{\begin{array}{ll}
	\Delta  h_{\varepsilon, j}=-\varepsilon & \text { in }\ B,\\
	h_{\varepsilon, j}=u_{j}+\varepsilon\geq u+\varepsilon & \text{ on } \partial B.
	\end{array}\right.$$
	
	It will be enough to show that $h_{\varepsilon, j}\geq u$ in $B$, since  then	
	$$u(x_1)\leq h_{\varepsilon, j}(x_1) =\frac{1}{\sigma^{n-1}(S^{n-1})}\int_{\partial B}\frac{u_{j}(x)}{r^{n-1}}\, d\sigma^{n-1}(x)+\varepsilon\left(1+\frac{r^2}{2n}\right).$$
	The latter expression, using monotone convergence, converges to 
	$$\frac{1}{\sigma^{n-1}(\partial B)}\int_{\partial B}u(x)\, d\sigma^{n-1}(x),$$ when $j\to\infty,\, \varepsilon\searrow 0$, and we obtain that $u$ satisfies the spherical mean value inequality. This yields subharmonicity.

	The set
	$$
	O=\left\{(x, y)\in\overline{B}\times\overline{B}:h_{\varepsilon, j}(x)-u(y)>\frac{\varepsilon}{2} \right\}
	$$
	is relatively open in $\overline{B}\times\overline{B}$ since $u$ is upper semicontinuous, and contains 
$$\{(x, x):x\in\partial B\}.$$

 Thus, the quantity $\Vert x-y\Vert^{4}$ is bounded away from zero on $\partial(B\times B)\backslash O$. Also, finite upper semicontinuous functions are bounded above on compact sets, and so we may choose $c>0$ large enough so that $w>0$ on $\partial(B\times B)$ , where
	$$
	w(x, y)=h_{\varepsilon, j}(x)-u(y)+c\Vert x-y\Vert^{4},\ x, y\in\overline{B}.
	$$
	We suppose, for the sake of contradiction, that the minimum value of the lower-semicontinuous function $w$ on $\overline{B}\times\overline{B}$ is attained at some interior point \newline $(x_{0}, y_{0})\in B\times B$. We have
		$$h_{\varepsilon, j}(x)-u(y)+c\Vert x-y\Vert^{4}\geq h_{\varepsilon, j}(x_{0})-u(y_{0})+c\Vert x_{0}-y_{0}\Vert^{4},\ x, y\in\overline{B}.$$
		Now we define
		$$
		\varphi(x):=h_{\varepsilon, j}(x_{0})+c(\Vert x_{0}-y_{0}\Vert^{4}-\Vert x-y_{0}\Vert^{4})\ ,\ x\in\overline{B}.
		$$ Setting $y=y_{0}$, we obtain the inequality
		$$ h_{\varepsilon, j}-\varphi\geq 0.$$
	Further,  $ h_{\varepsilon, j}-\varphi$ is $C^2$ and attains its minimum value at $x_{0}$, so  it's Hessian is non negative definite at $x_0$. Hence,
	$$\Delta \varphi(x_{0})\leq\Delta  h_{\varepsilon, j}(x_{0})=-\varepsilon.
	$$
	In particular, $x_{0}\neq y_{0}$ since $\Delta \varphi(y_{0})=0.$
	
	Similarly, if we define
	$$
	\psi(y):=u(y_{0})+c(\Vert x_{0}-y\Vert^{4}-\Vert x_{0}-y_{0}\Vert^{4})\ ,\ y\in\overline{B},
	$$
	the inequality $$h_{\varepsilon, j}(x)-u(y)+c\Vert x-y\Vert^{4}\geq h_{\varepsilon, j}(x_{0})-u(y_{0})+c\Vert x_{0}-y_{0}\Vert^{4},\ x, y\in\overline{B}$$
	transforms to
	$$u-\psi\leq 0,$$
	 by setting $x=x_{0}$. 
	 
	 Suppose that $y_0\in E_{1}$.	Since $ u-\psi$  and attains its maximum value $0$ at $y_{0}$, then   $$0=\nabla u(y_{0})=\nabla\psi(y_{0})\neq0,$$ because $x_{0}\neq y_{0}$ and the gradient of $\psi$ vanishes only there. But if  $y_0\not\in E_1$,  by hypothesis, the formula
	$$
	v(s)=w(x_{0}+s,\ y_{0}+s)=h_{\varepsilon, j}(x_{0}+s)-u(y_{0}+s)+c\Vert x_{0}-y_{0}\Vert^{4}
	$$
	defines a function which is superharmonic on some neighborhood of $0$ in $\mathbb{R}^{n}$. Since $v$ attains a local minimum at $0$, it must be constant near $0$. However, this leads to the contradictory conclusion that $u\in C^{\infty}$ and $\Delta  u=-\varepsilon<0$ near $y_{0}.$
	
	The theorem now follows, because
	$$
	\min_{x\in\overline{B}}(h_{\varepsilon, j}(x)-u(x))=\min_{x\in\overline{B}}w(x, x)\geq\min_{(x,y)\in\overline{B}\times\overline{B}}w(x,y)=\min_{(x,y)\in\partial(B\times B)}w(x,y)\geq 0.
	$$

	\end{proof}
   
The plurisubharmonic case is now easy. Having an upper semicontinuous function which is plurisubharmonic on some open neighborhood of $E_1$, we first extend is as a subharmonic  function on the whole of $\Omega$ by Theorem \ref{GardinerSodin}. Now this subharmonic function is actually plurisubharmonic by Theorem \ref{DDD} with $E_2=E_3=\emptyset$.

\section{Rad\'{o}- type theorem in the $C^{1,p}$ and $C^{1}$ case}
This section deals with the main theorem:
\begin{theorem}\label{rado}
 Let $\Omega$ be open in $\mathbb R^{n}$ (respectively in $\mathbb C^{n}$) and $E \subseteq \Omega$ be a Borel set.	If $u\in C^{1,p}(\Omega),\, p\in(0,1]$ is  subharmonic (respectively plurisubharmonic) in some open neighborhood of $\Omega\setminus E $ and the Hausdorff measure $\mathcal H^{p}(u(E))=0$   then $u$ is actually subharmonic (respectively plurisubharmonic) in $\Omega$. If $u\in C^{1}(\Omega)$ then the same conclusion holds if $u(E)$  is at most countable. The results are optimal with respect to the size of the image of $E $.
\end{theorem}
 \begin{remark} The assumption that $E$ is Borel may seem artificial. As $u$ is (pluri)sub-\newline harmonic on some open neighborhood $\Omega'$ of $\Omega\setminus E$, we could have taken the relatively closed set $\Omega\setminus\Omega'\subseteq E$ instead of $E$. Because $u$ is continuous, $u(E)$ is then a $F_\sigma$ set and $u^{-1}(u(E))$  is relatively $F_\sigma$ in $\Omega$. This means that we could have confined ourselves to just relatively closed or relatively $F_\sigma$ sets $E$. However, we hope  that similar ideas may be useful to study extension theorems  with minimal assumptions (less than Lipschitz but stronger in another directions), and without the continuity the above argument fails. It turns out that assuming $E$ is a Borel set is by no means a greater restriction than assuming it to be relatively closed in the discontinuous setting, whereas the failure of $E$  to be Borel or Souslin leads to considerable measure theoretic difficulties explained in  Remark \ref{Borel}. So we made this choice out of convenience.
 	\end{remark}
\begin{remark}\label{Borel}  In \cite{Kr}, where the analogous theorem is proved in the harmonic setting, the conditions are expressed somewhat more technically using inner Hausdorff measures to tackle supersets of  $u(E)$ which are not necessarily Hausdorff measurable. In our case this can be avoided by using the following seemingly not widely known facts.  As $E$ is Borel and $u$ is at least upper semicontinuous, hence a Borel mapping, $u(E)$ is Lebesgue measurable but not necessarily Borel, see \cite{B} Theorem 6.7.3. This is not enough to conclude that  $u(E)$ is also Hausdorff measurable, as even the continuous image of a Borel set is not necessarily Borel, and the $\sigma$- algebras of Lebesgue measurable and $\mathcal H^{p}$- measurable sets in $\mathbb R^{n}$ are in general different if $0<p<n$, as the example of a non-Lebesgue measurable set situated in some lower-dimensional subspace of $\mathbb R^{n}$ demonstrates. However, $u(E)$ is an analytic set (also known as Souslin set or ${\bf \Sigma_{1}^{1}}$- set in the projective hierarchy), as the Borel image of a Borel set, see Theorem 6.7.3 in \cite{B}. Analytic sets are measurable with respect to any Borel measure, see Corollary 2.12.7 or Theorem 7.4.1 \cite{B} or Theorem 26 in \cite{R}. As $\mathcal H^{p}$ is a Borel measure (Proposition 3.10.9 in \cite{B} or Theorem 27 in \cite{R}) it follows that $u(E)$ is $\mathcal H^{p}$- measurable. Now the property ''every compact subset has vanishing $\mathcal H^{p}$ measure'' is equivalent to being of zero $\mathcal H^{p}$ measure for analytic sets, see Corollary 2.10.48 in \cite{F} (this assertion is specific for $\mathbb R^n$).  Also the property ''every compact subset is at most countable'' from \cite{Kr} is equivalent to being at most countable for analytic sets, see Corollary 6.7.13 in \cite{B}. The latter is not true for arbitrary sets, as the so-called Bernstein sets demonstrate.\end{remark}
\begin{remark}
	The subharmonic part of Theorem \ref{rado} is essentially settled in \cite{GS}, as it can be easily deduced from there using only results which are already well-established. We provide the argument but we take no credit for it. What is substantially new is the plurisubharmonic  part of the theorem.
\end{remark}
\begin{proof}   We set $\Omega_0\subseteq \Omega$ to be the set where $\nabla u \neq 0$. It is clearly an open set which can further disconnect the connected components of $\Omega$. Also some part of $E $ can be contained in $\Omega_0$. We start with the $C^{1,p}$ case, as the continuously differentiable part requires a slightly different approach. As in Theorem $1$ in \cite{Kr}, it follows that $\Omega_0\cap E $ has zero Hausdorff measure of dimension  $n-1+p$ (respectively $2n-1+p$). This is essentially a consequence of the implicit function theorem. As $u$ is also subharmonic in $\Omega_0\setminus E $, we use  the extension result of \cite{SY} (see also \cite{P}), to conclude that $u$ is subharmonic on $\Omega_0$. In the $C^{1}$ case we have to use the subharmonic extension result \cite{Y}. In the plurisubharmonic case we can now use Theorem \ref{DDD}   to conclude that $u$ is plurisubharmonic in $\Omega_0$, as $E_3=E \cap \Omega_0$ is of Lebesgue measure zero.   It remains to extend the (pluri)subharmonicity through $E \cap(\Omega\setminus \Omega_0)$. But on this set one has $ \nabla u=0$ and, as $u\in C^{1}$- this follows form \cite{GS} in the subharmonic case. In the plurisubharmonic case we first extend $u$ as subharmonic function on $\Omega$ and use Theorem \ref{DDD} afterwards.
	
	The same examples as in \cite{Kr} demonstrate that for any $\varepsilon>0$ there is a relatively closed set $E \subseteq \Omega$ and a $C^{1,p}$ function $u$ such that $0<\mathcal H^{p}(u(E))<\varepsilon$, $u$ is not subharmonic on $\Omega$ but is subharmonic (actually locally affine) on $\Omega\setminus E $. The construction is sketched as follows. Let $G=u(E)$ and put
	$$
	\alpha(x):=\mathcal{H}^{p}(\{t\in G:t\leq x\})=\int_{-\infty}^{x}\chi_{G}(t)d \mathcal{H}^{p}(t)\ ,\ x\in \mathbb{R}
	$$	
	We fix an open interval $J$ with $\displaystyle \inf_{x\in J}\alpha(x)>0$ on which $\alpha$ is non constant and define
	$$
	\beta(t):=q^{2} \int_{-\infty}^{t}\alpha(x)dx,\quad t\in J,
	$$
	where the constant $q\geq 1$ is chosen in such a way that $q\alpha (x)\geq 1$ for $x\in J$. Then $\beta$ maps $J$ increasingly on an open interval $I$ and $\beta'\geq 1$ on $J$.  We denote by $\gamma:I\rightarrow J$ the corresponding inverse mapping $\gamma=\beta^{-1}$. 
Now $u(x_1,\ldots,x_n):=\gamma(x_1)$ is the required counterexample in the $C^{1,p}$ setting. The only thing we need to add to \cite{Kr} is that $\gamma$ is concave.

	 If $u(E)$ is uncountable then it contains a compact perfect set $H$. By Corollary 4 to Theorem 35 in \cite{R} there is a gauge function $h$ (also known as dimension function) such that the generalized Hausdorff measure (see \cite{R}) $\mathcal H^{h}(H)>0$. Then, as in Lemma 1 of \cite{Kr}, there exists a $C^1$ function $\gamma$, even such that the modulus of continuity of $\gamma'$ is $h$, defined on some open interval $I$ such that $\gamma'$ is locally affine on $I\setminus \gamma^{-1}(H)$ but not globally affine. Now $u(x_1,\ldots,x_n)=\gamma(x_1)$ is the required counterexample in the $C^{1}$ setting. 
	\end{proof}

In the next theorem we deal with more regular functions:
\begin{theorem}
Let $\Omega$ be open in $\mathbb R^{n}$ (respectively in $\mathbb C^{n}$) and $E\subseteq \Omega$ be a Borel set.	If $u\in C^{2}(\Omega)$ is  subharmonic (respectively plurisubharmonic) in some open neighborhood of $\Omega\setminus E$ then $u$ extends subharmonically (respectively plurisubharmonically) to the whole  $\Omega$ if and only if $u(E)$ has empty interior.
\end{theorem}
 
\begin{proof}
	It is clear that every $C^2$ function which is (pluri)subharmonic outside $E $ is actually (pluri)subharmonic through $E $ if and only if $E $ has empty interior, as the Laplacian (respectively the complex Hessian) is continuous and non negative on the dense complement of $E $. As above, we define $\Omega_0$ to be the noncritical set of $u$. As $u:\Omega_0\to\mathbb R$ is a submersion, it is a continuous open mapping. Now, assuming that $\mathbb R\setminus u(E)$ is dense, $u^{-1}(\mathbb R\setminus u(E))\cap\Omega_0$  is contained in  $\Omega_0\setminus E $ and dense in $\Omega_0$, hence $u$ is (pluri)subharmonic on $\Omega_0$. Now $u$ extends also through $E \cap(\Omega\setminus\Omega_0)$, as $\nabla u=0$ there. The conclusion is that for $C^2$ functions the extension result holds if and only if $u(E)$ has dense complement, for if $u(E)$ has nonempty interior it contains an interval $(a,b)$ and a counterexample  can be produced as follows. Let  $u(z)=|z|^4-2|z|^2+1$, $E =\left\{z\in\mathbb C: |z|\leq \frac{1}{\sqrt{2}}\right\}$, $u(E)=\left[\frac{1}{4},1\right]$. Let $\alpha:\mathbb  R\to\mathbb R$ be an affine increasing function, such that $\alpha(\left[\frac{1}{4},1\right]) \subseteq (a,b)$. Now $u=\alpha\circ u$ is an evident counterexample.
	
	The same is of course true when $u$ is more regular than $C^2$. 
\end{proof}

\section{Applications}
In this last section we briefly collect some direct corollaries to harmonic and pluriharmonic extension problems.

By applying Theorem \ref{rado} to the subharmonic $u$ and superharmonic $-u$ we recover the following special case of the theorem of Kr\'{a}l \cite{Kr}:

 \begin{theorem}
 	Let $\Omega$ be open in $\mathbb R^{n}$  and $E\subseteq \Omega$ be a relatively closed set.\newline	If $u\in C^{1,p}(\Omega),\, p\in(0,1]$ is  harmonic  in $\Omega\setminus E$ and the Hausdorff measure\newline $\mathcal H^{p}(u(E))=0$  then $u$ is actually harmonic  in $\Omega$. If $u\in C^{1}(\Omega)$ then the same conclusion holds if $u(E)$  is at most countable. 
 \end{theorem}

Likewise, by applying Theorem \ref{rado}  to the plurisubharmonic $u$ and plurisuperharmonic $-u$ we get the following extension result for pluriharmonic functions:

\begin{theorem}
	Let $\Omega$ be open in $\mathbb C^{n}$  and $E\subseteq \Omega$ be a relatively closed set.	\newline If $u\in C^{1,p}(\Omega),\, p\in(0,1]$ is  pluriharmonic  in $\Omega\setminus E$ and the Hausdorff measure $\mathcal H^{p}(u(E))=0$  then $u$ is actually pluriharmonic  in $\Omega$. If $u\in C^{1}(\Omega)$ then the same conclusion holds if $u(E)$  is at most countable. 

\end{theorem}

\bibliographystyle{amsplain}

\end{document}